\documentclass[12pt,reqno]{amsart}

\usepackage{amsmath,amssymb,amsfonts,amsthm}
\usepackage[abbrev]{amsrefs}
\usepackage{hyperref}
\usepackage{tikz-cd}

\usepackage{microtype}
\usepackage{color,hyperref}
\definecolor{darkblue}{rgb}{0.0,0.0,0.3}
\hypersetup{colorlinks,breaklinks,
            linkcolor=red,urlcolor=red,
            anchorcolor=red,citecolor=red}

\theoremstyle{plain}
\newtheorem{theorem}{Theorem}  
\newtheorem{cor}[theorem]{Corollary}
\newtheorem{lemma}[theorem]{Lemma}

\newtheorem{prop}[theorem]{Proposition}
\numberwithin{equation}{section}

\theoremstyle{definition}
\newtheorem{defi}[theorem]{Definition}

\newtheorem{exam}[theorem]{Example}
\numberwithin{theorem}{section}
\newcommand{\m}[1]{\mathbb{#1}}

\newcommand{\mr}[1]{\mathrm{#1}}
\newcommand{\mc}[1]{\mathcal{#1}}

\newcommand{\TT}{{\m{T}}}

\newcommand{\ZZ}{{\m{Z}}}
\newcommand{\FF}{{\m{F}}}

\newcommand{\CC}{{\m{C}}}

\newcommand{\s}{\subseteq}

\newcommand{\mP}{\mc{P}}

\newcommand{\mU}{\mc{U}}

\newcommand{\Aut}{\mr{Aut}}

\newcommand{\conv}{\mr{conv}}

\newcommand{\ov}{\overline}
\newcommand{\sm}{\setminus}

\begin{document}

\title[Reduced twisted crossed products over C*-simple groups]{Reduced twisted crossed products \\ over C*-simple groups}
\author[R. S. Bryder]{Rasmus Sylvester Bryder}
\address{Department of Mathematics\\ University of Copenhagen\\
Universitetsparken 5, 2100 Copenhagen, Denmark}
\email{bryder@math.ku.dk}

\author[M. Kennedy]{Matthew Kennedy}
\address{Department of Pure Mathematics\\ University of Waterloo\\
Waterloo, ON, N2L 3G1, Canada}
\email{matt.kennedy@uwaterloo.ca}

\begin{abstract}
We consider reduced crossed products of twisted C*-dynamical systems over C*-simple groups. We prove there is a bijective correspondence between maximal ideals of the reduced crossed product and maximal invariant ideals of the underlying C*-algebra, and a bijective correspondence between tracial states on the reduced crossed product and invariant tracial states on the underlying C*-algebra. In particular, the reduced crossed product is simple if and only if the underlying C*-algebra has no proper non-trivial invariant ideals, and the reduced crossed product has a unique tracial state if and only if the underlying C*-algebra has a unique invariant tracial state. We also show that the reduced crossed product satisfies an averaging property analogous to Powers' averaging property.
\end{abstract}

\dedicatory{Dedicated to the memory of Uffe Haagerup}

\thanks{The first author is supported by a PhD stipend from the Danish National Research Foundation (DNRF) through the Centre for Symmetry and Deformation at University of Copenhagen}
\thanks{The second author is supported by a grant from the Canadian National Sciences and Engineering Research Council (NSERC)}
\maketitle

\section{Introduction}
A discrete group $G$ is said to be \emph{C*-simple} if its reduced C*-algebra $C^*_r(G)$ is simple, meaning that it has no non-trivial closed two-sided ideals. In representation-theoretic terms, this is equivalent to the statement that every unitary representation of $G$ that is weakly contained in the left regular representation is weakly equivalent to the left regular representation. In this paper we consider (twisted) dynamical systems over C*-simple groups. 

The first examples of C*-simple groups were discovered in 1975 by Powers \cite{powersfreegroup}. He showed that non-abelian free groups are C*-simple. Specifically, he showed that the reduced C*-algebra of a non-abelian free group satisfies a strong kind of averaging property that, in particular, implies the simplicity of the algebra. Later, by isolating and abstracting the essential properties needed in Powers' proof, these ideas were applied to establish the C*-simplicity of many groups, and the class of C*-simple groups is now known to be quite large (see e.g. \cite{delaharpesurvey} for a survey). For example, it includes all acylindrically hyperbolic \cite{dahmaniguiradelosin} and linear \cite{poznansky} groups with trivial amenable radical.

Recent work has led to a much better understanding of C*-simplicity. In \cite{kalantarkennedy} Kalantar and the second author established a  characterization of C*-simplicity in terms of the dynamics of the action of a group on its Furstenberg boundary. A systematic study of C*-simplicity making use of this perspective was subsequently undertaken by Breuillard, Kalantar, Ozawa and the second author in \cite{kalantarkennedyozbr}. More recently, Haagerup \cite{haagerup2015} and the second author \cite{kennedy2015char} independently showed that the C*-simplicity of a group is equivalent to the reduced C*-algebra satisfying Powers' averaging property. The second author also established an ``intrinsic'' characterization of the C*-simplicity of a group in terms of its subgroup structure. We also point out the work of Le Boudec \cite{leboudec} and Raum \cites{raum1,raum2}.

An important fact about (twisted) C*-dynamical systems over C*-simple groups is that they are often much more tractable than arbitrary (twisted) C*-dynamical systems. This was first observed by de la Harpe and Skandalis in \cite{delaharpeskandalis} for C*-dynamical systems over Powers' groups. Specifically, they showed that if $G$ is a Powers' group and $A$ is a C*-algebra equipped with an action of $G$, then the reduced crossed product $A \rtimes_r G$ is simple provided that $A$ has no proper non-trivial $G$-invariant closed ideals.

The results of de la Harpe and Skandalis were later generalized to weak Powers' groups by Boca and Ni{\c{t}}ic{\u{a}} in \cite{bocanitica}, and to $(\mr{P_{com}})$ groups by Bekka, Cowling and de la Harpe in \cite{bekkacowlingdelaharpe2}. More recently in \cite{kalantarkennedyozbr}, the results were shown to hold for all C*-dynamical systems over C*-simple groups.

Twisted C*-dynamical systems over C*-simple groups were first studied by B\'{e}dos in \cite{bedosart}. He established results analogous to the results of de la Harpe and Skandalis for twisted C*-dynamical systems over weak Powers' groups. More recently in \cite{bedosconti}, B\'{e}dos and Conti obtained more general results for exact twisted C*-dynamical systems over $(\mr{P_{com}})$ and $PH$ groups. Specifically, they showed that if $G$ is a $(\mr{P_{com}})$ or $PH$ group and $A$ is a unital C*-algebra equipped with a twisted action $(\alpha,\sigma)$ of $G$, then there is a bijective correspondence between maximal closed ideals of the reduced crossed product $A \rtimes_{\alpha,r}^\sigma G$ and maximal $G$-invariant closed ideals of $A$.

In this paper we will show that these results hold for all twisted C*-dynamical systems over C*-simple groups. In particular, we will not require the dynamical system to be exact. Throughout, by a twisted C*-dynamical system we mean a quadruple $(A,G,\alpha,\sigma)$ consisting of a unital C*-algebra $A$, a discrete group $G$ and a twisted action $(\alpha,\sigma)$ of $G$ on $A$. The corresponding reduced crossed product is a C*-algebra that encodes the system (see Section \ref{preliminaries} for details). In particular, every C*-dynamical system can be viewed as a twisted C*-dynamical system, and our results are new even in this case.

\begin{theorem}
Let $(A,G,\alpha,\sigma)$ be a twisted C*-dynamical system over a C*-simple group $G$. Then there is a bijective correspondence between maximal $G$-invariant ideals of $A$ and maximal ideals of the reduced crossed product $A\rtimes_{\alpha,r}^\sigma G$.
\end{theorem}

The previous result immediately implies the following corollary.

\begin{cor}
Let $(A,G,\alpha,\sigma)$ be a twisted C*-dynamical system over a C*-simple group $G$. Then the reduced crossed product $A\rtimes_{\alpha,r}^\sigma G$ is simple if and only if $A$ has no proper non-trivial $G$-invariant ideals.
\end{cor}

Our results on twisted C*-dynamical systems have some interesting applications. For example, we provide another proof of the fact that the class of C*-simple groups is stable under extensions, originally proved in \cite{kalantarkennedyozbr}. In the same paper, it was shown that for an arbitrary discrete group $G$, every tracial state on the reduced C*-algebra of $G$ is supported on the amenable radical $R_a(G)$ of $G$. It follows that if $R_a(G)$ is trivial, and in particular if $G$ is C*-simple, then the reduced C*-algebra has a unique tracial state. We obtain an analogous result for tracial states on reduced crossed products of twisted C*-dynamical systems.

\begin{theorem}
Let $(A,G,\alpha,\sigma)$ be a twisted C*-dynamical system over a discrete group $G$ with amenable radical $R_a(G)$. For every tracial state $\tau$ on the reduced crossed product $A \rtimes_{\alpha,r}^\sigma G$, $\tau = \tau \circ E_{R_a(G)}$, where $E_{R_a(G)}$ denotes the canonical conditional expectation from $A \rtimes_{\alpha,r}^\sigma G$ to $A \rtimes_{\alpha,r}^\sigma R_a(G)$.
\end{theorem}

The previous result immediately implies the following corollary.

\begin{cor}
Let $(A,G,\alpha,\sigma)$ be a twisted C*-dynamical system over a discrete group $G$. Suppose that $G$ is C*-simple or, more generally, that the amenable radical of $G$ is trivial. Then there is a bijective correspondence between tracial states on the reduced crossed product $A \rtimes_{\alpha,r}^\sigma G$ and $G$-invariant tracial states on $A$. Thus $A \rtimes_{\alpha,r}^\sigma G$ has a unique tracial state if and only if $A$ has a unique $G$-invariant tracial state.
\end{cor}

The paper is arranged as follows. In Section 2, we establish preliminaries on twisted C*-dynamical systems as well as the needed knowledge on the Furstenberg boundary. Section 3 is dedicated to finding conditions for reduced twisted crossed products to have Powers' averaging property (a condition stronger than the well-known Dixmier property). In Section 4 we determine the maximal ideal structure of a reduced twisted crossed product when the underlying group is C*-simple. Lastly, in Section 5 we find a correspondence between $G$-invariant tracial states on the C*-algebra $A$ in a twisted C*-dynamical system $(A,G,\alpha,\sigma)$ and tracial states on the reduced twisted crossed product $A\rtimes_{\alpha,r}^\sigma G$, whenever $G$ has the unique trace property.

We will employ standard notation. If $G$ is a group, the neutral element of $G$ will be denoted by $1$. If $A$ is a unital C*-algebra, then $\mU(A)$ will denote the unitary group of $A$ and $\Aut(A)$ is the group of all $^*$-automorphisms of $A$. For any Hilbert space $H$, $B(H)$ will denote the C*-algebra of bounded linear operators on $H$.

\section{Preliminaries}\label{preliminaries}
\subsection{Twisted C*-dynamical systems}

Throughout this paper, we will consider twisted C*-dynamical systems $(A,G,\alpha,\sigma)$. This means that $A$ is a unital C*-algebra with unit $1$, $G$ is a discrete group and $\alpha\colon G\to\Aut( A)$ and $\sigma\colon G\times G\to\mU( A)$ are maps satisfying the identities
\begin{align*}\alpha_r\circ\alpha_s&=\mr{Ad}(\sigma(r,s))\circ\alpha_{rs},\\\sigma(r,s)\sigma(rs,t)&=\alpha_r(\sigma(s,t))\sigma(r,st),\\\sigma(s,1)&=\sigma(1,s)=1,\end{align*} for all $r,s,t\in G$. Here $\alpha_r=\alpha(r)\in\Aut(A)$ for all $r\in G$, and $\mr{Ad}(v)\in\Aut(A)$ is the inner automorphism of $A$ given by $\mr{Ad}(v)(a)=vav^*$ for all $a\in A$, where $v\in\mU(A)$. The tuple $(\alpha,\sigma)$ is referred to as a \emph{twisted action} of $G$ on $A$, and $\sigma$ is called the \emph{normalized $2$-cocycle} (or just the \emph{cocycle}) of the twisted action. Whenever $A=\CC$, then $\alpha$ is the trivial map and we say that $\sigma\colon G\times G\to\TT$ is a \emph{multiplier}. If $\sigma$ is \emph{trivial}, i.e., $\sigma(r,s)=1$ for all $r,s\in G$, then $\alpha$ is a group homomorphism and $(A,G,\alpha)$ is a C*-dynamical system in the usual sense.

Assuming that $A$ is faithfully represented on some Hilbert space $H$, we may realize the \emph{reduced twisted crossed product} $A\rtimes_{\alpha,r}^\sigma G$ (cf. \cite{packerraeburn}, \cite{quigg}, \cite{zellermeier}) as the C*-subalgebra of $B( H\otimes\ell^2(G))$ generated by the set of operators $\{\pi_\alpha(a),\,\lambda_\sigma(g)\,|\,a\in A,\,g\in G\}$ where $\pi_\alpha\colon A\to B(H\otimes\ell^2(G))$ is the faithful representation of $A$ given by $$\pi_\alpha(a)(\xi\otimes\delta_t)=\alpha_{t^{-1}}(a)\xi\otimes\delta_t,\quad a\in A,\ \xi\in H,\ t\in G,$$ and $\lambda_\sigma(g)\in\mU(H\otimes\ell^2(G))$ is a unitary operator for $g\in G$ given by
$$\lambda_\sigma(g)(\xi\otimes\delta_t)=\sigma(t^{-1}g^{-1},g)\xi\otimes\delta_{gt},\quad t\in G,\ \xi\in H.$$ Moreover, it is well-known that $A\rtimes_{\alpha,r}^\sigma G$ does not depend on the choice of faithful representation $ A\subset B( H)$ \cite[p. 552]{quigg}. Henceforth, we will identify $\pi_\alpha(A)$ with $A$.

Straightforward computation shows that the maps $\pi_\alpha$ and $\lambda_\sigma$ define a \emph{covariant representation} \cite[p. 135]{zellermeier} of the twisted C*-dynamical system $(A,G,\alpha,\sigma)$, i.e., the identities
\begin{align*}\alpha_s(a)=\lambda_\sigma(s)a\lambda_\sigma(s)^*,\quad\lambda_\sigma(s)\lambda_\sigma(t)=\sigma(s,t)\lambda_\sigma(st)\end{align*} hold for all $a\in A$ and $s,t\in G$. It follows that the conjugation map $\alpha'(s)=\mr{Ad}(\lambda_\sigma(s))$ maps $A\rtimes_{\alpha,r}^\sigma G$ onto itself for all $s\in G$. Viewing $\sigma(r,s)$ as a unitary operator in $A\rtimes_{\alpha,r}^\sigma G$ for $r,s\in G$, we may then note that $(A\rtimes_{\alpha,r}^\sigma G,G,\alpha',\sigma)$ is a twisted C*-dynamical system such that the inclusion $A\to A\rtimes_{\alpha,r}^\sigma G$ is $G$-equivariant, i.e., $\alpha'_s|_A=\alpha_s$ for all $s\in G$.

For any twisted C*-dynamical system $(A,G,\alpha,\sigma)$ there exists a faithful conditional expectation $E$ of $A\rtimes_{\alpha,r}^\sigma G$ onto $A$ such that $E(\lambda_\sigma(s))=0$ for all $g\in G\sm\{1\}$ (cf. \cite[Theorem 2.2]{bedosart}, \cite[Th\'{e}or\`{e}me 4.12]{zellermeier}). Moreover, $E$ is $G$-equivariant, i.e., $$E(\lambda_\sigma(g)x\lambda_\sigma(g)^*)=\alpha_g(E(x)),\quad g\in G,\ x\in A\rtimes_{\alpha,r}^\sigma G.$$ Henceforth, $E$ is referred to as the \emph{canonical conditional expectation} of $A\rtimes_{\alpha,r}^\sigma G$ onto $A$.

A particularly good reason for working with twisted C*-dynamical systems is that their crossed products can be realized as iterated crossed products with respect to a normal subgroup and the corresponding quotient group. A proof of the following result can be found in \cite[Theorem 2.1]{bedosart} (for the case of full twisted crossed products, see \cite[Theorem 4.1]{packerraeburn}).
\begin{theorem}\label{rcpstructure}Let $(A,G,\alpha,\sigma)$ be a twisted C*-dynamical system and let $N$ be a normal subgroup of $G$. If we denote the restriction of the twisted action $(\alpha,\sigma)$ to $N$ also by $(\alpha,\sigma)$, there exists a twisted action $(\gamma,\nu)$ of the quotient group $G/N$ on the reduced twisted crossed product $A\rtimes_{\alpha,r}^\sigma N$ such that
$$A\rtimes_{\alpha,r}^\sigma G\cong (A\rtimes_{\alpha,r}^\sigma N)\rtimes_{\gamma,r}^\nu(G/N).$$
\end{theorem}

\subsection{Boundary actions.} Let $G$ be a discrete group. We say that a compact Hausdorff space $X$ is a \emph{$G$-space} if $G$ acts by homeomorphisms on $X$. In this case, we will write $gx$ for the image of $x$ under $g$ for any $x\in X$ and $g\in G$.

The space of probability measures on a compact Hausdorff space $X$ will always be denoted by $\mP(X)$. Note that $\mP(X)$ can be identified with the state space of the C*-algebra $C(X)$ of continuous complex-valued functions on $X$. If $X$ is a $G$-space, there is a canonical $G$-action on $C(X)$ given by
$$(sf)(x)=f(s^{-1}x),\quad s\in G,\ f\in C(X),\ x\in X.$$
Note that if we define $L\colon G\to\mr{Aut}(C(X))$ by $L_sf=sf$ for $s\in G$ and $f\in C(X)$, then $(C(X),G,L)$ is a C*-dynamical system. For any C*-dynamical system $(A,G,\alpha)$ we may define a $G$-action on a weak$^*$-compact $G$-invariant subset $X$ of $A^*$ by
$$(s\phi)(a)=\phi(s^{-1}a),\quad s\in G,\ \phi\in A^*,\ a\in A,$$
so that $X$ is a $G$-space. In this way, the space $\mP(X)$ becomes a $G$-space.

The action of $G$ on a $G$-space $X$ is said to be \emph{minimal} if every $G$-orbit in $X$ is dense in $X$, and \emph{strongly proximal} if it holds for any $\mu\in\mP(X)$ that the weak$^*$-closure of the $G$-orbit $G\mu$ contains a point mass $\delta_x$ for some $x\in X$. A $G$-space for which the $G$-action is minimal and strongly proximal is called a \emph{$G$-boundary}. There always exists a unique $G$-boundary $\partial_FG$ called the \emph{Furstenberg boundary} \cite[Section 4]{furstenberg} which is universal in the sense that for any $G$-boundary $X$, there exists a surjective $G$-equivariant continuous map $\partial_FG\to X$.

Let $X$ be a $G$-boundary. If $(A,G,\alpha,\sigma)$ is a twisted C*-dynamical system, we will frequently consider the twisted C*-dynamical system $(A \otimes C(X), G, \beta, \tau)$ obtained by defining $\beta\colon G\to\Aut(A\otimes C(\partial_FG))$ and $\tau \colon G\times G\to\mU(A\otimes C(X))$ by $$\beta_s(a\otimes f)=\alpha_s(a)\otimes(sf),\quad \tau(r,s)=\sigma(r,s)\otimes 1$$ for $r,s\in G$, $a\in A$ and $f\in C(X)$. The inclusion map $A\to A\otimes C(X)$ given by $a\mapsto a\otimes 1$ allows us to identify $A\rtimes_{\alpha,r}^\sigma G$ as a unital $G$-invariant C*-subalgebra of $(A \otimes C(X)) \rtimes_{\beta,r}^\tau G$, so that the unitaries $\lambda_\sigma(g)$ and $\lambda_\tau(g)$ are identified in $(A \otimes C(X)) \rtimes_{\beta,r}^\tau G$ for all $g\in G$. We will refer to $(A \otimes C(X), G, \beta, \tau)$ as a {\em natural extension} of $(A, G, \alpha, \sigma)$.

The following lemma will be useful in the subsequent sections.

\begin{lemma}\label{lemmaconv} Let $(A,G,\alpha,\sigma)$ be a twisted C*-dynamical system and let $X$ be a $G$-boundary. The natural extension $(A \otimes C(X), G, \beta, \tau)$ has the following property: For any states $\phi_1,\ldots,\phi_n$ on the reduced crossed product $A\rtimes_{\alpha,r}^{\sigma} G$ and any point $x \in X$, there exist states $\psi_1,\ldots,\psi_n$ on the reduced crossed product $(A\otimes C(X))\rtimes_{\beta,r}^{\tau} G$ and a net $(g_j)$ in $G$ such that
\[
\psi_i|_{A\rtimes_{\alpha,r}^{\tau} G} = \lim_j \phi_i \circ \operatorname{Ad}(\lambda_\sigma(g_j)) \quad \text{and} \quad \psi_i|_{C(X)}=\delta_x, \quad 1 \leq i \leq n,
\]
the limit being taken in the weak$^*$ topology.\end{lemma}

\begin{proof}
Extend the states $\phi_1,\ldots,\phi_n$ on $A\rtimes_{\alpha,r}^\sigma G$ to states $\hat{\phi}_1,\ldots,\hat{\phi}_n$ respectively on $(A \otimes C(X)) \rtimes_{\beta,r}^\tau G$. For each $i$, let $\mu_i$ denote the probability measure obtained by restricting $\hat{\phi}_i$ to $C(X)$, and let $\mu = \frac{1}{n} \sum_{i = 1}^n \mu_i$.

The twisted action $(\beta,\tau)$ restricted to $C(X)$ is the simply the $G$-action on $C(X)$, so strong proximality yields a net $(g_j)$ in $G$ such that $g_j\mu\to\delta_x$ in the weak$^*$ topology. By compactness, we may assume that $(g_j\mu_i)$ converges for all $i$, in which case we have $g_j \mu_i \to \delta_x$ for each $i$ as $\delta_x$ is an extreme point of $\mP(X)$.

Compactness allows us to further assume that for each $i$, the net $(\hat{\phi}_i \circ \operatorname{Ad}(\lambda_\tau(g_j)))$ converges to a state $\psi_i$ on $(A \otimes C(X)) \rtimes_{\beta,r}^\tau G$ in the weak$^*$ topology. Then by construction, $\psi_i|_{C(X)}=\delta_x$ and
\[
\psi_i|_{A\rtimes_{\alpha,r}^\sigma G} = \lim_j (\hat{\phi}_i \circ \operatorname{Ad}(\lambda_\tau(g_j))) |_{A\rtimes_{\alpha,r}^\sigma G} = \lim_j \phi_i \circ \operatorname{Ad}(\lambda_\sigma(g_j)). \qedhere
\]
\end{proof}

\section{Powers' averaging property}

The following definition can be seen as a generalization of \cite[Definition 5.2]{kennedy2015char}.

\begin{defi} \label{defi:powers-avg-process}
A twisted C*-dynamical system $(A,G,\alpha,\sigma)$ is said to have {\em Powers' averaging property} if for every element $b$ in the reduced twisted crossed product $A\rtimes_{\alpha,r}^\sigma G$ satisfying $E(b)=0$ and every $\epsilon > 0$ there are $g_1,\ldots,g_n \in G$ such that
\[
\left\| \frac{1}{n} \sum_{i=1}^n \lambda_\sigma(g_i) b \lambda_\sigma(g_i)^* \right\| < \epsilon.
\]
Here $E\colon A\rtimes_{\alpha,r}^\sigma G \to A$ denotes the canonical conditional expectation of $A\rtimes_{\alpha,r}^\sigma G$ onto $A$.
\end{defi}

A key result of B\'{e}dos and Conti \cite[Theorem 3.8]{bedosconti} is that reduced crossed products of twisted C*-dynamical systems over $(\mr{P_{com}})$ and $PH$ groups satisfy an averaging property that is similar to Definition \ref{defi:powers-avg-process}. In this section we will show that twisted C*-dynamical systems over C*-simple groups always satisfy Powers' averaging property.

Recall that the $G$-action on a $G$-space $X$ is said to be \emph{free} if $gx\neq x$ for all $g\in G\sm\{1\}$ and $x\in X$. One of the main results of \cite{kalantarkennedy} is the following theorem.

\begin{theorem}\label{kk2014}A discrete group $G$ is C*-simple if and only if the $G$-action on the Furstenberg boundary $\partial_FG$ is free.\end{theorem}

With the initial requirement that we consider only C*-simple groups, we will first prove some preliminary results on linear functionals of reduced twisted crossed products.

\begin{lemma}\label{lemmaconv1}
Let $(A,G,\alpha,\sigma)$ be a twisted C*-dynamical system where $G$ is C*-simple, and let $\phi$ be a bounded linear functional on $A\rtimes_{\alpha,r}^\sigma G$. If $E$ denotes the canonical conditional expectation of $A\rtimes_{\alpha,r}^\sigma G$ onto $A$, then there exists $\psi$ in the weak$^*$ closure of $\{\phi \circ \operatorname{Ad}(\lambda_\sigma(g)) \mid g \in G\}$ such that $\psi=\psi\circ E$.
\end{lemma}

\begin{proof}

By the Hahn-Jordan decomposition for bounded linear functionals on C*-algebras, we can write $\phi = c_1 \phi_1 - c_2 \phi_2 + i(c_3 \phi_3 - c_4 \phi_4)$ for real numbers $c_1,\ldots,c_4 \geq 0$ and states $\phi_1,\ldots,\phi_4$ on $A\rtimes_{\alpha,r}^\sigma G$.

Consider the natural extension $(A \otimes C(\partial_FG), G, \beta, \tau)$ of $(A,G,\alpha,\sigma)$. By Lemma \ref{lemmaconv}, we can find states $\psi_1,\psi_2,\psi_3,\psi_4$ on $(A\otimes C(\partial_FG))\rtimes_{\beta,r}^{\tau} G$, a point $x \in \partial_FG$ and a net $(g_j)$ in $G$ such that
\[
\psi_i|_{A\rtimes_{\alpha,r}^{\tau} G} = \lim_j \phi_i \circ \operatorname{Ad}(\lambda_\sigma(g_j)) \quad \text{and} \quad \psi_i|_{C(X)}=\delta_x, \quad i=1,2,3,4.
\]

Since $G$ is C*-simple, Theorem \ref{kk2014} implies that the $G$-action on $\partial_FG$ is free, so for any $g \in G \sm \{e\}$ there is $f \in C(\partial_FG)$ such that $f(x) = 1$ and $f(g^{-1}x) = 0$. It follows that for each $i$ and every $a \in A$,
\begin{align*}
\psi_i|_{A\rtimes_{\alpha,r}^\sigma G}(a\lambda_\sigma(g))&=\psi_i(a\lambda_\tau(g))f(x)\\
&=\psi_i((a\otimes 1)\lambda_{\tau}(g)(1\otimes f))\\
&=\psi_i((a\otimes 1)(1\otimes (gf))\lambda_{\tau}(g))\\
&=\psi_i((1\otimes (gf))(a\otimes 1)\lambda_{\tau}(g))\\
&=f(g^{-1}x)\psi_i((a\otimes 1)\lambda_{\tau}(g))\\&=0.
\end{align*}
It follows by continuity that $\psi_i|_{A\rtimes_{\alpha,r}^\sigma G} = \psi_i|_{A\rtimes_{\alpha,r}^\sigma G} \circ E$ for each $i$. Hence we can take the restriction of $\psi = c_1\psi_1 - c_2 \psi_2 + i(c_3\psi_3 - c_4\psi_4)$ to $A\rtimes_{\alpha,r}^\sigma G$.
\end{proof}

\begin{theorem}\label{convtwist}
If $(A,G,\alpha,\sigma)$ is a twisted C*-dynamical system where $G$ is C*-simple, and $E$ denotes the canonical conditional expectation of $A\rtimes_{\alpha,r}^\sigma G$ onto $A$, then
\[
0\in\ov{\conv} \{ \lambda_\sigma(g) x \lambda_\sigma(g)^* \mid g \in G \}
\]
for all $x\in A\rtimes_{\alpha,r}^\sigma G$ satisfying $E(x)=0$, the closure being in the norm. In particular, $(A,G,\alpha,\sigma)$ has Powers' averaging property.
\end{theorem}

\begin{proof}
Fix $x \in A\rtimes_{\alpha,r}^\sigma G$ satisfying $E(x) = 0$ and suppose for the sake of contradiction that the claim does not hold for this $x$. Let
\[
K = \ov{\conv} \{ \lambda_\sigma(g) x \lambda_\sigma(g)^* \mid g \in G \}.
\]
By the Hahn-Banach theorem, there exists a bounded linear functional $\phi$ on $A\rtimes_{\alpha,r}^\sigma G$ such that $\inf_{y \in K}\mr{Re}(\phi(y))>0$. But in particular, this implies that $\inf_{g \in G} |\phi \circ \operatorname{Ad}(\lambda_\sigma(g))(x)| > 0$, which contradicts Lemma \ref{lemmaconv1}.
\end{proof}

\section{Ideal structure of reduced crossed products}

In this section we consider the ideal structure of reduced crossed products of twisted C*-dynamical systems. In general, it is not possible to relate the ideal structure of a reduced crossed product to the ideal structure of its underlying C*-algebra, even when the group is C*-simple. Indeed, de la Harpe and Skandalis \cite{delaharpeskandalis} constructed examples of C*-dynamical systems over Powers groups (which are C*-simple) with the property that the reduced crossed product has many non-trivial ideals, but the underlying C*-algebra has only a single non-trivial invariant ideal.

However, B\'{e}dos and Conti \cite{bedosconti}, showed that for exact twisted C*-dynamical systems over $(\mr{P_{com}})$ and $PH$ groups, there is a bijective correspondence between maximal ideals of the reduced crossed product and maximal ideals of the underlying C*-algebra.

In this section we will show that this bijective correspondence between maximal ideals holds for all twisted C*-dynamical systems over C*-simple groups. In particular, we will not require the system to be exact.

Let $(A,G,\alpha,\sigma)$ be a twisted C*-dynamical system, and let $I$ be a $G$-invariant ideal of $A$. Let $I\rtimes_{\alpha,r}^\sigma G$ denote the ideal in $A\rtimes_{\alpha,r}^\sigma G$ generated by $I$, and let $\pi\colon A \to A/I$ denote the corresponding quotient map. The twisted action $(\alpha,\sigma)$ then induces a twisted action $(\dot\alpha,\dot\sigma)$ of $G$ on $A/I$ such that $\dot\alpha_s\circ\pi=\pi\circ\alpha_s$ for all $s\in G$ and $\dot\sigma=\pi\circ\sigma$. Therefore $\pi$ induces a surjective $^*$-homomorphism $\pi \rtimes_{\alpha,r}^\sigma \operatorname{id}\colon A \rtimes_{\alpha,r}^\sigma G \to A/I \rtimes_{\dot\alpha,r}^{\dot\sigma} G$ at the level of crossed products, giving rise to the following commutative diagram:
\begin{equation}
\begin{tikzcd}
0 \arrow{r} & I \rtimes_{\alpha,r}^\sigma G \arrow{r}\arrow{d}{E_I} & A \rtimes_{\alpha,r}^\sigma G \arrow{r}{\pi \rtimes_{\alpha,r}^\sigma \operatorname{id}} \arrow{d}{E_A} & A/I \rtimes_{\dot\alpha,r}^{\dot\sigma} G \arrow{r}\arrow{d}{E_{A/I}} & 0\\
0 \arrow{r} & I \arrow{r} & A \arrow{r}{\pi} & A/I \arrow{r} & 0
\end{tikzcd}\label{eqn:comm-diagram}\end{equation}
Here $E_A$ and $E_{A/I}$ denote the canonical conditional expectations and $E_I=E_A|_{I\rtimes_{\alpha,r}^\sigma G}$. We may note that the ideal $I\rtimes_{\alpha,r}^\sigma G$ coincides with the reduced twisted crossed product of the twisted C*-dynamical system $(I,G,\alpha,\sigma)$ where $\sigma$ is a cocycle in the multiplier algebra of $I$.

In the following, we will use the notation
$$I\bar{\rtimes}_{\alpha,r}^\sigma G=\ker(\pi\rtimes_{\alpha,r}^\sigma\mr{id}).$$
Now observe that the upper sequence in (\ref{eqn:comm-diagram}) is exact precisely when $I \rtimes_{\alpha,r}^\sigma G = I\bar{\rtimes}_{\alpha,r}^\sigma G.$ It is clear that the inclusion $I \rtimes_{\alpha,r}^\sigma G \subset I\bar{\rtimes}_{\alpha,r}^\sigma G$ always holds, but equality does not necessarily hold in general (see e.g. \cite[Remark 1.17]{sierakowski2010}). If equality does hold for every $G$-invariant ideal $I$ in $A$, then the C*-dynamical system $(A,G,\alpha,\sigma)$ is said to be {\em exact}. If $G$ is exact, then a result of Exel \cite{exel2002} implies that every twisted C*-dynamical system over $G$ is exact.

The next lemma generalizes \cite[Lemma 3.20]{kalantarkennedyozbr}.

\begin{lemma} \label{lem:ideals-1}
Let $(A,G,\alpha,\sigma)$ be a twisted C*-dynamical system where $G$ is C*-simple and let $(A \otimes C(\partial_FG),G,\beta,\tau)$ denote the natural extension. Let $I$ be a proper ideal in $A \rtimes_{\alpha,r}^\sigma G$ and let $J$ denote the ideal in $(A \otimes C(\partial_FG)) \rtimes_{\beta,r}^\tau G$ generated by $I$. Then $J$ is proper.
\end{lemma}

\begin{proof}
Let $\phi$ be a state on $A \rtimes_{\alpha,r}^\sigma G$ such that $\phi(I) = 0$. By Lemma \ref{lemmaconv} there is a state $\psi$ on $(A \otimes C(\partial_FG)) \rtimes_{\beta,r}^\tau G$, a net $(g_i)$ in $G$ and $x \in \partial_FG$ such that $\psi|_{A\rtimes_{\alpha,r}^{\sigma}G} = \lim_j \phi \circ \operatorname{Ad}(\lambda_\sigma(g_j))$ and $\psi|_{C(\partial_FG)} = \delta_x$.

Note that $\psi|_{A\rtimes_{\alpha,r}^{\tau} G}(I) = 0$ and $C(\partial_FG)$ is contained in the multiplicative domain of $\psi$. Hence for $b \in I$, $a_1,a_2\in A$, $f_1,f_2 \in C(\partial_FG)$ and $s_1,s_2\in G$ we have
\begin{align*}
&\psi((a_1\otimes f_1)\lambda_\tau(s_1)b(a_2\otimes f_2)\lambda_\tau(s_2))\\&=f_1(x)\psi(a_1\lambda_\sigma(s_1)ba_2\lambda_\sigma(s_2))f_2(s_2x)\\
&=0.
\end{align*}
It follows that $\psi(J) = 0$. Hence $J$ is proper.
\end{proof}

The next lemma generalizes \cite[Lemma 3.21]{kalantarkennedyozbr}. The proof is a simple modification of the proof given there.

\begin{lemma} \label{lem:ideals-2}
Let $(A,G,\alpha,\sigma)$ be a twisted C*-dynamical system where $G$ is C*-simple and let $(A \otimes C(\partial_FG),G,\beta,\tau)$ denote the natural extension. Let $J$ be an ideal in $(A \otimes C(\partial_FG)) \rtimes_{\beta,r}^\tau G$. Then setting $J_A = J \cap (A \otimes C(\partial_FG))$,
\[
J_A \rtimes_{\beta,r}^\tau G \subset J \subset J_A \bar{\rtimes}_{\beta,r}^\tau G.
\]
\end{lemma}

For any twisted $C^*$-dynamical system $(A,G,\alpha,\sigma)$ and any $G$-in\-va\-ri\-ant ideal $I$ in $A\bar\rtimes_{\alpha,r}^\sigma G$, the commutative diagram (\ref{eqn:comm-diagram}) yields the identity
\begin{equation}\label{kereq1}
(I\bar\rtimes_{\alpha,r}^\sigma G)\cap A=I.
\end{equation}
Moreover, let $X$ be a $G$-boundary. We observe for the natural extension $(A\otimes C(X),G,\beta,\tau)$ that if $K\s (A\otimes C(X))\rtimes_{\beta,r}^\tau G$ is an ideal and $K_A=K\cap(A\otimes C(X))$, then there is a commutative diagram of $^*$-homomorphisms
\[
\begin{tikzcd}
A\rtimes_{\alpha,r}^\sigma G\arrow{r}\arrow{d}&(A\otimes C(X))\rtimes_{\beta,r}^\tau G\arrow{d}\\
A/(K\cap A)\rtimes_{\dot\alpha,r}^{\dot\sigma} G\arrow{r}&(A\otimes C(X))/K_A\rtimes_{\dot\beta,r}^{\dot\tau} G
\end{tikzcd}
\]
where the horizontal arrows are injective. It follows that 
\begin{equation}\label{kereq2}(K_A\bar\rtimes_{\beta,r}^\tau G)\cap(A\rtimes_{\alpha,r}^\sigma G)=(K\cap A)\bar\rtimes_{\alpha,r}^\sigma G.\end{equation}

In the following, for any twisted $C^*$-dynamical system $(A,G,\alpha,\sigma)$, a \emph{maximal $G$-invariant ideal} $I$ in $A$ will mean a proper $G$-invariant ideal which is maximal among proper $G$-invariant ideals in $A$. Note that this does not entail that $I$ is a maximal ideal (for instance, for non-trivial $G$ one may consider $A=C(X)$ where $X$ is a minimal $G$-space with more than one point).

\begin{theorem}
Let $(A,G,\alpha,\sigma)$ be a twisted C*-dynamical system where $G$ is C*-simple. For a maximal ideal $I$ of $A \rtimes_{\alpha,r}^\sigma G$, $I \cap A$ is a maximal $G$-invariant ideal of $A$. Conversely, for a maximal $G$-invariant ideal $Y$ of $A$, the ideal $Y \bar{\rtimes}_{\alpha,r}^\sigma G$ of $A \rtimes_{\alpha,r}^\sigma G$ is maximal. Moreover, this correspondence is bijective.

\end{theorem}

\begin{proof}
Let $Y$ be a maximal $G$-invariant ideal in $A$. We must show that the ideal $Y \bar{\rtimes}_{\alpha,r}^\sigma G$ in $A \rtimes_{\alpha,r}^\sigma G$ is maximal. Suppose that $J$ is a proper ideal in $A \rtimes_{\alpha,r}^\sigma G$ such that $Y \bar{\rtimes}_{\alpha,r}^\sigma G \subset J$.

Consider the natural extension $(A \otimes C(\partial_FG), G, \beta, \tau)$ of $(A,G,\alpha,\sigma)$. Let $K$ denote the ideal in $(A \otimes C(\partial_FG)) \rtimes_{\beta,r}^\tau G$ generated by $J$. By Lemma \ref{lem:ideals-2}, $K \subset K_A \overline{\rtimes}_{\beta,r}^\tau G$, where $K_A = K \cap (A \otimes C(\partial_FG))$. By (\ref{kereq2}),
\[
J \subset K \cap (A \rtimes_{\alpha,r}^\sigma G) \subset (K_A \bar{\rtimes}_{\beta,r}^\tau G) \cap (A \rtimes_{\alpha,r}^\sigma G)=(K \cap A) \bar{\rtimes}_{\alpha,r}^\sigma G,
\]
and applying (\ref{kereq1}) to $Y$ and $K\cap A$ gives
\[
Y \subset J \cap A \subset K \cap A.
\]
Since $J$ is proper, Lemma \ref{lem:ideals-1} implies that $K$ is proper, so the maximality of $Y$ implies that $Y = K \cap A$ since $K\cap A$ is $G$-invariant. From above, $J \subset Y \bar{\rtimes}_{\alpha,r}^\sigma G$, and it follows that $Y \bar{\rtimes}_{\alpha,r}^\sigma G$ is maximal.

Now let $I$ be a maximal ideal in $A \rtimes_{\alpha,r}^\sigma G$. We must show that the ideal $I \cap A$ is maximal among proper $G$-invariant ideals in $A$.

Let $J$ denote the ideal in $(A \otimes C(\partial_FG)) \rtimes_{\beta,r}^\tau G$ generated by $I$. By Lemma \ref{lem:ideals-2}, $J \subset J_A \bar{\rtimes}_{\beta,r}^\tau G$, where $J_A = J \cap (A \otimes C(\partial_FG))$. Hence by (\ref{kereq2})
\[
I \subset J \cap (A \rtimes_{\alpha,r}^\sigma G) \subset (J_A \bar{\rtimes}_{\beta,r}^\tau G) \cap (A \rtimes_{\alpha,r}^\sigma G) = (J \cap A)\bar{\rtimes}_{\alpha,r}^\sigma G.
\]
Since $I$ is proper, Lemma \ref{lem:ideals-1} implies that $J \cap A$ is proper in $A$, so the maximality of $I$ implies that $I = (J \cap A)\bar{\rtimes}_{\alpha,r}^\sigma G$. Hence $I \cap A = J \cap A$ by (\ref{kereq1}), and it follows that 
\begin{equation}\label{kermaxeq}I = (I \cap A) \bar{\rtimes}_{\alpha,r}^\sigma G.\end{equation}

Now suppose that $Z$ is a proper $G$-invariant ideal in $A$ such that $I \cap A \subset Z$. Then $Z \bar{\rtimes}_{\alpha,r}^\sigma G$ is a proper ideal in $A \rtimes_{\alpha,r}^\sigma G$ and
\[
I = (I \cap A) \bar{\rtimes}_{\alpha,r}^\sigma G \subset Z \bar{\rtimes}_{\alpha,r}^\sigma G,
\]
so the maximality of $I$ implies that $I = Z \bar{\rtimes}_{\alpha,r}^\sigma G$. Hence
\[
I \cap A = (Z \bar{\rtimes}_{\alpha,r}^\sigma G) \cap A = Z,
\]
and it follows that $I \cap A$ is maximal. Finally it follows from the identities (\ref{kereq1}) and (\ref{kermaxeq}) that the correspondence is bijective.
\end{proof}

For a twisted $C^*$-dynamical system $(A,G,\alpha,\sigma)$ we say that $A$ is \emph{$G$-simple} if the only $G$-invariant ideals in $A$ are $\{0\}$ and $A$. As a corollary we obtain the following generalization of \cite[Theorem 3.19]{kalantarkennedyozbr}.

\begin{cor}\label{twistsimple2}Let $(A,G,\alpha,\sigma)$ is a twisted C*-dynamical system where $G$ is C*-simple. Then $A\rtimes_{\alpha,r}^\sigma G$ is simple if and only if $A$ is $G$-simple.\end{cor}
\begin{cor}\label{twistsimple3}If $G$ is C*-simple, then the reduced twisted group C*-algebra $C^*_r(G,\sigma)$ is simple for every multiplier $\sigma\colon G\times G\to\TT$.\end{cor}

By appealing to the previously mentioned structure theorem for twisted reduced crossed products, it is possible to say something about twisted C*-dynamical systems whenever the underlying group has a C*-simple quotient.

\begin{cor}\label{normalsimple}Let $(A,G,\alpha,\sigma)$ be a twisted C*-dynamical system and let $N$ be a normal subgroup of $G$. Continue to write $(\alpha,\sigma)$ for the restriction of $(\alpha,\sigma)$ to $N$. If $G/N$ is C*-simple, then $A\rtimes_{\alpha,r}^\sigma G$ is simple whenever $A\rtimes_{\alpha,r}^\sigma N$ is simple.\end{cor}
\begin{proof}By Theorem \ref{rcpstructure} there exists a twisted action $(\gamma,\nu)$ of $G/N$ on $A\rtimes_{\alpha,r}^\sigma N$ such that $$A\rtimes_{\alpha,r}^\sigma G\cong(A\rtimes_{\alpha,r}^\sigma N)\rtimes_{\gamma,r}^\nu(G/N).$$ The desired conclusion now follows from Corollary \ref{twistsimple2}.\end{proof}

We also obtain another proof of \cite[Theorem 3.14]{kalantarkennedyozbr}.

\begin{cor}\label{groupext}Let $G$ be a discrete group. If $G$ contains a normal C*-simple subgroup $N$ such that $G/N$ is C*-simple, then $G$ is C*-simple.\end{cor}

What follows is a generalization of \cite[Proposition 3.13]{bedosconti}. The proof there applies verbatim, so we omit it.

\begin{prop}\label{idealtrivialact}Let $(A,G,\alpha,\sigma)$ be an exact twisted C*-dynamical system where $\alpha$ is trivial. If $(A,G,\alpha,\sigma)$ has Powers' averaging property, then there is a bijective correspondence between the set of ideals in $A$ and the set of ideals in $A\rtimes_{\alpha,r}^\sigma G$ given by $Y\mapsto Y\bar\rtimes_{\alpha,r}^\sigma G$.\end{prop}

In the next example, we will consider one application of the above results, adapted from \cite[Example 4.4]{bedosconti}.

\begin{exam}Let $G$ be a discrete group with center $Z$. By \cite[Theorem 2.1]{bedosart} there exists a twisted action $(\alpha,\sigma)$ of $G/Z$ on $C^*_r(Z)$ such that
$$C^*_r(G)\cong C^*_r(Z)\rtimes_{\alpha,r}^\sigma (G/Z),$$
where we may choose $\alpha\colon G/Z\to\mr{Aut}(C^*_r(Z))$ to be the trivial map. If $G/Z$ is exact and $C^*$-simple, then $(C^*_r(Z),G/Z,\alpha,\sigma)$ is exact and has Powers' averaging property by Theorem \ref{convtwist} so that Proposition \ref{idealtrivialact} applies. Hence the ideals of $C^*_r(G)$ are in one-to-one correspondence with the ideals of $C^*_r(Z)$, i.e., the open (or closed) subsets of the dual group $\hat{Z}$ of $Z$.

For instance, for $n\geq 3$ the braid group $B_n$ with $n$ generators has center $Z_n\cong\ZZ$, and B\'{e}dos proved in \cite[p. 536]{bedosart} that $B_n/Z_n$ is C*-simple. Further, we have short exact sequences
\begin{align*}
1 \to \FF_{n-1} \to P_n/Z_n \to P_{n-1}/Z_{n-1} \to 1,\\
1 \to P_n/Z_n \to B_n/Z_n \to S_n \to 1
\end{align*}
where $\FF_{n-1}$ is the free non-abelian group of $n-1$ generators, $P_n$ is the pure braid group on $n$ generators and $S_n$ is the symmetric group on $n$ generators \cite[Proposition 6]{giordanodlh}. We recall that free groups and finite groups are exact, and that extensions of exact groups are exact \cite[Proposition 5.1.11]{brownozawa}. As $P_2=Z_2$ we may thus conclude by induction that $P_n/Z_n$ is exact, so that $B_n/Z_n$ is exact. Therefore the ideals of $C^*_r(B_n)$ are in one-to-one correspondence with the open (or closed) subsets of $\TT$.\end{exam}

\section{Tracial states on reduced twisted crossed products}

\noindent In this section we will relate tracial states on reduced crossed products of twisted C*-dynamical systems to tracial states on the underlying C*-algebra.

The following lemma is well-known, and we omit the proof.
\begin{lemma}Let $(A,G,\alpha,\sigma)$ be a twisted C*-dynamical system and let $E$ denote the canonical conditional expectation of $A\rtimes_{\alpha,r}^\sigma G$ onto $A$. If $\tau$ is a $G$-invariant trace on $A$, then $\tau\circ E$ is a trace on $A\rtimes_{\alpha,r}^\sigma G$.\end{lemma}

We recall that any discrete group $G$ has a largest amenable normal subgroup $R_a(G)$, called the \emph{amenable radical} of $G$. Any amenable normal subgroup of $G$ is then contained in $R_a(G)$. In \cite[Theorem 4.1]{kalantarkennedyozbr} it is proved that the canonical trace $\tau$ on the reduced group C*-algebra $C^*_r(G)$ satisfies $\tau(\lambda(g))=0$ for all $g\notin R_a(G)$. As it turns out, the argument therein easily extends to reduced twisted crossed products of $G$.

\begin{theorem} \label{amrad}
Let $(A,G,\alpha,\sigma)$ be a twisted C*-dynamical system over a discrete group $G$ with amenable radical $R_a(G)$. For every tracial state $\tau$ on the reduced crossed product $A \rtimes_{\alpha,r}^\sigma G$, $\tau = \tau \circ E_{R_a(G)}$, where $E_{R_a(G)}$ denotes the canonical conditional expectation from $A \rtimes_{\alpha,r}^\sigma G$ to $A \rtimes_{\alpha,r}^\sigma R_a(G)$.
\end{theorem}

\begin{proof}
We must show that for any tracial state $\phi$ on $A\rtimes_{\alpha,r}^\sigma G$ and any $a \in A$, we have $\phi(\lambda_\sigma(g))=0$ for all $a\in A$ and $g\notin R_a(G)$. We consider the natural extension $(A\otimes C(\partial_FG),G,\beta,\tau)$. By Lemma \ref{lemmaconv} and the $G$-invariance of $\phi$, there is a state $\psi$ on $(A \otimes C(\partial_FG)) \rtimes_{\beta,r}^\tau G$, a net $(g_i)$ in $G$ and $x \in \partial_FG$ such that $\psi|_{A\rtimes_{\alpha,r}^{\sigma}G} = \lim_j \phi \circ \operatorname{Ad}(\lambda_\sigma(g_j))=\phi$ and $\psi|_{C(\partial_FG)} = \delta_x$. 

For $g\notin R_a(G)$, there exists $y\in\partial_F G$ such that $g^{-1}y\neq y$ \cite[Proposition 7]{furman}. By minimality of the $G$-action on $\partial_FG$ there exists a net $(h_i)$ in $G$ such that $h_ix\to y$. By compactness, we may assume that $(\psi\circ\mr{Ad}\lambda_{\sigma'}(h_i))$ converges to a state $\eta$ on $(A \otimes C(\partial_FG)) \rtimes_{\beta,r}^\tau G$ which, by the $G$-invariance of $\phi$, then satisfies $\eta|_{A\rtimes_{\alpha,r}^{\sigma}G}=\phi$ and $\eta|_{C(\partial_F G)}=\delta_y$. As $C(\partial_FG)$ is contained in the multiplicative domain of $\eta$, taking a function $f\in C(\partial_FG)$ such that $f(g^{-1}y)=0$ and $f(y)=1$ now yields $\phi(a\lambda_\sigma(g))=0$ for any $a \in A$, just as in the proof of Lemma \ref{lemmaconv1}.
\end{proof}

The next result was previously shown in \cite[Corollary 3.9]{bedosconti} for the case when $G$ is a $(\mr{P}_{\mr{com}})$ or $PH$ group.

\begin{cor} \label{tracecor}
Let $(A,G,\alpha,\sigma)$ be a twisted C*-dynamical system over a discrete group $G$. Suppose that $G$ is C*-simple or, more generally, that the amenable radical of $G$ is trivial. For every $G$-invariant tracial state $\tau$ on $A$, $\tau \circ E$ is a tracial state on the reduced twisted crossed product $A \rtimes_{\alpha,r}^\sigma G$, where $E$ denotes the canonical conditional expectation of $A \rtimes_{\alpha,r}^\sigma G$ onto $A$. Conversely, every tracial state on $A \rtimes_{\alpha,r}^\sigma G$ arises in this way from a $G$-invariant tracial state on $A$. Moreover, this correspondence is bijective. Thus $A \rtimes_{\alpha,r}^\sigma G$ has a unique tracial state if and only if $A$ has a unique $G$-invariant tracial state.
\end{cor}
 
\begin{proof}By Theorem \ref{amrad}, it follows that $\tau=\tau\circ E$ for any tracial state $\tau$ on $A\rtimes_{\alpha,r}^\sigma G$. Any tracial state $\tau$ on $A\rtimes_{\alpha,r}^\sigma G$ is therefore uniquely determined by its restriction to $A$. The claim now follows.\end{proof}

\subsection*{Acknowledgements}
The present work was initiated during the first author's stay at the University of Waterloo in the fall of 2015. The first author would like to thank the aforementioned institution for its generosity and hospitality.

Both authors are grateful to Erik B\'{e}dos and Narutaka Ozawa for their helpful comments and suggestions.

\bibliography{bib}

\end{document}